\documentclass{llncs}

\usepackage{research}

\begin{document}

\title{A Markov Chain Analysis of a Pattern Matching Coin Game}

\titlerunning{Pattern Matching Card Games}

\author{James Brofos}

\authorrunning{James Brofos}
\tocauthor{James Brofos}

\institute{Dartmouth College, Hanover NH 03755, USA\\
\email{james.a.brofos.15@dartmouth.edu},\\ Website:
\href{http://www.cs.dartmouth.edu/~james/}{http://www.cs.dartmouth.edu/$\sim$james/}}

\maketitle

\begin{abstract}
	In late May of 2014 I received an email from a colleague introducing to me a non-transitive game developed by Walter Penney. This paper explores this probability game from the perspective of a coin tossing game, and further discusses some similarly interesting properties arising out of a Markov Chain analysis. In particular, we calculate the number of ``rounds'' that are expected to be played in each variation of the game by leveraging the fundamental matrix. Additionally, I derive a novel method for calculating the probabilistic advantage obtained by the player following Penney's strategy. I also produce an exhaustive proof that Penney's strategy is optimal for his namesake game.
\end{abstract}

\section{The Penney Ante Game: Flips and Tricks}

The game developed by Walter Penney is not particularly well known and we think that half the fun is to be found in playing the game with your friends. Therefore we would like to take this first foray into the development of the game as an introduction and a lesson. The game is classically not played with suited cards (as is its treatment in \cite{nishiyama}), but instead with a fair coin. We discuss Penney's probabilistic game in the light of the coin.

Imagine that two players, \opp and \p, are to play Penney's game. Then we may assume, without loss of generality, that \opp has challenged \p, and for that reason it falls first to \play to elect a ``binary'' (though it is perhaps more natural to consider it either ``heads'' or ``tails'') sequence of length three. It's easy to see that there are eight such sequences, and let us suppose that \play has (lucklessly) chosen $\set{HHH}$ -- that is, the sequence of three heads one after the other.

In response, \opp also chooses a sequence. Let us pointedly suppose that \opp chooses the sequence $\set{THH}$. Then the game proceeds as follows:

\begin{quotation}
	A coin is flipped repeatedly until either the first or second player's sequence is observed. In that case, the player whose sequence is found first is declared the winner, and we say that this player has won the ``trick.'' As many tricks as one likes may be played, and the player who wins the most tricks wins overall. 
    
    For example, the following sequence may result from a fair coin,
    \begin{equation}
    	\set{HHTTTHTTHHHTHTH}
    \end{equation}
    Then we can by inspection confirm that the sequence chosen by \opp occurs at the eighth 3-sequence (immediately preceding \playp elected 3-sequence). Then we have that \opp has won this particular trick.
\end{quotation}

What makes this game particularly interesting is that there exists an \emph{optimal} strategy that may be taken by the second player (in this case \o) such that they may gain a probabilistic advantage over the first player (here, \p). We formalize this notion in the following definition.

\begin{definition}[Optimality]
	We say that a 3-sequence of binary random variables is a strategy. Furthermore, such a strategy $s_1$ is optimal for another strategy $s_2$ iff the probability that the sequence $s_1$ will occur before the sequence $s_2$ (and thereby winning the trick) is larger than $\frac{1}{2}$. We denote this probability using the notation,
    \begin{equation}
    	\Pr{s_1 \big| s_2}
    \end{equation}
\end{definition}

\subsection{Non-Transitive Games}

The word ``non-transitive'' may at first be rather intimidating, but as it happens almost everyone has been met with a non-transitive game at one point in their lives or another. We begin by defining a non-transitive relationship.

\begin{definition}[Non-Transitivity]
	A relationship between strategies $s_1$, $s_2$, and $s_3$ is said to be non-transitive if $s_1$ is optimal for $s_2$ and $s_2$ is optimal for $s_3$, but $s_1$ is not optimal for $s_3$.
\end{definition}

This is quite easy to find in everyday life, especially among young children! If one considers carefully, then it should be apparent that the classical game of rock-paper-scissors constitutes a non-transitive game. 

\begin{example}[Rock-Paper-Scissors]
	As we are playing a slightly different game in rock-paper-scissors, we shall for this section only consider a \emph{strategy} to be one of the game's namesake possibilities. Then we know that rock is optimal for scissors, and that scissors is optimal for paper. But if rock-paper-scissors were a transitive game, then rock would be optimal for paper. Indeed, this is not the case, and in fact just the opposite since paper is optimal for rock! Therefore, rock-paper-scissors is a non-transitive game.
\end{example}

Similarly, the Penney ante game is also a non-transitive game, which is what makes it interesting. Unlike rock-paper-scissors, however, where rock \emph{always} beats scissors, it is conceivable in the Penney ante game that a sub-optimal strategy may still win! Thus, optimality must be defined in terms of probabilities rather than in terms of absolute properties. 

\begin{proposition}
	The Penney ante game is a non-transitive game.
\end{proposition}
\begin{proof}
	As will follow from exhaustive proof later in the paper, we will have that $\set{HHT}$ has an optimal counterpart in $\set{THH}$. Furthermore, $\set{TTH}$ is optimal for $\set{THH}$. In a transitive game, this would suggest that $\set{TTH}$ is optimal for $\set{HHT}$, but indeed they are not, for each element in the strategy is simply the ``negation'' of the corresponding element of the other. Thus, if one strategy were sub-optimal, one could make it optimal by relabeling the sides of the coin.
    
    Interestingly, Penney's game exhibits what we suppose might be called a four-directional non-transitive relationship. Continuing where we left off, we will demonstrate later that the strategy $\set{HTT}$ is optimal for $\set{TTH}$. But crucially we have at last that $\set{HHT}$ is optimal for $\set{HTT}$. This is a result precisely opposite of what one would expect in a transitive game.
\end{proof}

\subsection{An Optimality Theorem}

\begin{lemma}
	\label{lemma:hhh_response}
	Assume that \play elects the sequence $\set{HHH}$ and that \opp in response chooses the sequence $\set{THH}$. Then \play will win iff $\set{HHH}$ comprises the first sequence to appear.
\end{lemma}

\begin{proof}
    The direction $\Leftarrow$ is quite immediate as if the first sequence to appear is $\set{HHH}$ then \play has won by definition. 
    
    The $\Rightarrow$ direction requires closer examination. Suppose that the first sequence was not $\set{HHH}$. Then we have the implication that there must therefore be at least one $T$ within the first 3-sequence. Then we have the following four possibilities for the next 2-sequence following the $T$:
    \begin{description}
    	\item[$\set{HH}$:] In this case, the total 3-sequence is $\set{THH}$, and so \opp wins the trick.
        \item[$\set{HT}$:] Since the sequence ends on a $T$, nothing has changed since we began, and we may take the sequence to ``reset'' in a sense on the last occurring $T$. 
        \item[$\set{TH}$:] From here, the sequence can either transition to a $H$ (in which case \opp wins the trick), or to a $T$, in which case we can apply the logic in the point above.
        \item[$\set{TT}$:] In this case, nothing has changed, so we may take the sequence to reset on the last $T$ again.
    \end{description}
    Notice that in each of the four exhaustive conditions above, none allow for the possibility of \play winning the trick. Therefore, we conclude that \play cannot win if a $T$ has occurred within the first 3-sequence. Therefore, \play can only win when the first three sequence does not contain $T$, or in other words, if the first 3-sequence is $\set{HHH}$. \qed
    
\end{proof}

\begin{corollary}
	The probability that \play wins given that he selects the sequence $\set{HHH}$ and that \opp plays optimally is,
    \begin{equation}
    	\Pr{\set{HHH} \big| \set{THH}} = \arg{\frac{1}{2}}^3 = \frac{1}{8}
    \end{equation}
    Incidentally, the probability that \opp wins under these conditions is $1 - \frac{1}{8}=\frac{7}{8}$.
\end{corollary}

\begin{theorem}[Penney's Optimal Strategy]
	\label{theorem:optimal}

	Denote by $\mathcal{B}$ a Bernoulli random variable with parameter $\frac{1}{2}$. Then the sequence $\set{\mathcal{B}_1, \mathcal{B}_2, \mathcal{B}_3}$ is a binary sequence of length three, which is analogous to the sequence chosen by \play in Penney's game. Then an optimal strategy for \opp to assume is the sequence,
    \begin{equation}
    	\set{\neg \mathcal{B}_2, \mathcal{B}_1, \mathcal{B}_2}
    \end{equation}
    Where $\neg$ is the logical negation such that,
    \begin{equation}
    	\neg : \set{0, 1} \to \set{1, 0}.
    \end{equation}

\end{theorem}

\begin{proof}[Part One of \Cref{theorem:optimal}]
    The proof is fortunately very straightforward, and it requires only the confirmation that \oppp sequence has consistently a higher probability of occurring first than does \p's. The eight possible strategies can be enumerated completely. 
    
    From \Cref{lemma:hhh_response} it is immediately clear that $\set{THH}$ is optimal for $\set{HHH}$. 
    
    If we assume that \play takes the strategy $\set{HHT}$, then \Cref{theorem:optimal} predicts that the strategy $\set{THH}$ is optimal. We can confirm this by calculating the probability that $\set{THH}$ will occur before the sequence $\set{HHT}$.
    \begin{eqnarray}
    	\Pr{\set{\scriptstyle{HHT}} \big| \set{\scriptstyle{THH}}} & = & \Pr{\set{\scriptstyle{HHT}}} + \Pr{\set{\scriptstyle{HHHT}}} + \Pr{\set{\scriptstyle{HHHHT}}} + \ldots\\
       	& = & \frac{1}{8} + \frac{1}{16} + \frac{1}{32} + \ldots = \sum_{k = 3}^\infty \arg{\frac{1}{2}}^k = \frac{\frac{1}{8}}{1 - \frac{1}{2}} = \frac{1}{4}
    \end{eqnarray}
    Therefore, we have via complementary events that,
    \begin{equation}
    	\Pr{\set{THH} \big| \set{HHT}} = 1 - \Pr{\set{HHT} \big| \set{THH}} = 1 - \frac{1}{4} = \frac{3}{4}
    \end{equation}
    This confirms that $\set{THH}$ is an optimal strategy for $\set{HHT}$. This concludes the cases for which we provide a proof for the moment. \qed

\end{proof}

Let us next examine the case where \play has elected the 3-sequence $\set{HTH}$ and \opp chooses $\set{HHT}$. Since both strategies depend on an initial result $H$, we may assume without loss of generality that the first coin flip in the sequence yields $H$. This proof is most easily understood via a Markov Chain argument, so we diverge here briefly to discuss Markov Chains and their relationship to Penney's game. 

\begin{lemma}
	Let \P~be the transition matrix of a Markov chain. Let the rows and columns of the transition matrix correspond to particular 3-sequences arising in the Penney ante game. Then the probability of transition from the 3-sequence $s_i$ to the sequence $s_j$ is given by the entry $\P_{i,j}$.
\end{lemma}

\begin{theorem}[Fundamental Theorem of Markov Chains]
	\label{theorem:markov}
	Let \P~be the transition matrix of a Markov Chain and let the row-vector $\mbox{\boldmath $x$}_0$ denote a probability vector of the state distributions initially. Then the probability of arriving in the 3-sequence $s_i$ after $n$ iterations of the game is,
    \begin{equation}
    	\mbox{\boldmath $x$}_n = \mbox{\boldmath $x$}_0\P^n
    \end{equation}
    In fact, the states may be generalized beyond 3-sequences to any number of states existing in an abstract system. It is this property that makes the theorem so useful in applications.
\end{theorem}

\begin{lemma}
	\label{lemma:markov_penney}
	We can now demonstrate how Markov Chains can be leveraged to address two further cases in Penney's game. In particular, $\set{HHT}$ is optimal for both $\set{HTH}$ and $\set{HTT}$.
\end{lemma}
\begin{proof}
    In the first case, we consider $\set{HTH}$. Then \Cref{theorem:optimal} predicts that \opp should choose the strategy $\set{HHT}$. In this case, we wish to understand $\Pr{\set{HHT} \big| \set{HTH}}$. Notice that the transition matrix, \P, for this game may be written,
    \begin{equation}
       	\left(\begin{array}{ccccccccc}
        	                        & \set{\scriptstyle{HHH}} & \set{\scriptstyle{HHT}} & \set{\scriptstyle{HTH}} & \set{\scriptstyle{HTT}} & \set{\scriptstyle{THH}} & \set{\scriptstyle{THT}} & \set{\scriptstyle{TTH}} & \set{\scriptstyle{TTT}} \\
			\set{\scriptstyle{HHH}} & \frac{1}{2} & \frac{1}{2} & 0 & 0 & 0 & 0 & 0 & 0 \\
			\set{\scriptstyle{HHT}} & 0 & 1 & 0 & 0 & 0 & 0 & 0 & 0 \\
			\set{\scriptstyle{HTH}} & 0 & 0 & 1 & 0 & 0 & 0 & 0 & 0 \\
			\set{\scriptstyle{HTT}} & 0 & 0 & 0 & 0 & 0 & 0 & \frac{1}{2} & \frac{1}{2} \\
			\set{\scriptstyle{THH}} & \frac{1}{2} & \frac{1}{2} & 0 & 0 & 0 & 0 & 0 & 0 \\
            \set{\scriptstyle{THT}} & 0 & 0 & \frac{1}{2} & \frac{1}{2} & 0 & 0 & 0 & 0 \\
            \set{\scriptstyle{TTH}} & 0 & 0 & 0 & 0 & \frac{1}{2} & \frac{1}{2} & 0 & 0 \\
            \set{\scriptstyle{TTT}} & 0 & 0 & 0 & 0 & 0 & 0 & \frac{1}{2} & \frac{1}{2}
        \end{array}\right)
    \end{equation}
    Then it may be demonstrated computationally that we have asymptotically the following behavior of $\lim_{n\to\infty} \P^n$:
    \begin{equation}
    	\left(\begin{array}{ccccccccc}
        	                        & \set{\scriptstyle{HHH}} & \set{\scriptstyle{HHT}} & \set{\scriptstyle{HTH}} & \set{\scriptstyle{HTT}} & \set{\scriptstyle{THH}} & \set{\scriptstyle{THT}} & \set{\scriptstyle{TTH}} & \set{\scriptstyle{TTT}} \\
			\set{\scriptstyle{HHH}} & 0 & 1 & 0 & 0 & 0 & 0 & 0 & 0 \\
			\set{\scriptstyle{HHT}} & 0 & 1 & 0 & 0 & 0 & 0 & 0 & 0 \\
			\set{\scriptstyle{HTH}} & 0 & 0 & 1 & 0 & 0 & 0 & 0 & 0 \\
			\set{\scriptstyle{HTT}} & 0 & \frac{2}{3} & \frac{1}{3} & 0 & 0 & 0 & 0 & 0 \\
			\set{\scriptstyle{THH}} & 0 & 1 & 0 & 0 & 0 & 0 & 0 & 0 \\
            \set{\scriptstyle{THT}} & 0 & \frac{1}{3} & \frac{2}{3} & 0 & 0 & 0 & 0 & 0 \\
            \set{\scriptstyle{TTH}} & 0 & \frac{2}{3} & \frac{1}{3} & 0 & 0 & 0 & 0 & 0 \\
            \set{\scriptstyle{TTT}} & 0 & \frac{2}{3} & \frac{1}{3} & 0 & 0 & 0 & 0 & 0
        \end{array}\right)
    \end{equation}
    Let us denote $\mbox{\boldmath $M$} = \lim_{n\to\infty} \P^n$. Notice then that each of initial states are equally likely, each occurring with probability $\frac{1}{8}$. Therefore, \Cref{theorem:markov} provides,
    \begin{equation}
    	\left\{\frac{1}{8}, \frac{1}{8}, \frac{1}{8}, \frac{1}{8}, \frac{1}{8}, \frac{1}{8}, \frac{1}{8}, \frac{1}{8}\right\} \mbox{\boldmath $M$} = \left\{0, \frac{2}{3}, \frac{1}{3}, 0, 0, 0, 0, 0\right\}
    \end{equation}
    This shows that $\Pr{\set{HHT} \big| \set{HTH}} = \frac{2}{3}$ and that therefore, \oppp strategy is optimal for \playp under these circumstances.
    
    To demonstrate that $\set{HHT}$ is optimal for $\set{HTT}$, we proceed simply in the same manner. We do not reproduce here the calculations as we did in the previous case, but they are easy to check using a programming language such as MATLAB or Python. Let \Q~be the transition matrix for the Markov Chain corresponding to the game where $\set{HTT}$ is \playp strategy and $\set{HTH}$ is \oppp sequence. Then we have that,
    \begin{equation}
    \left\{0,\frac{2}{3},0\frac{1}{3},0,0,0,0\right\} = \left\{\frac{1}{8}, \frac{1}{8}, \frac{1}{8}, \frac{1}{8}, \frac{1}{8}, \frac{1}{8}, \frac{1}{8}, \frac{1}{8}\right\}\left(\lim_{n\to\infty} \Q^n\right)
    \end{equation}
    It follows from here that,
    \begin{equation}
    	\Pr{\set{HTH} \big| \set{HTT}} = \frac{2}{3}
    \end{equation}
    In turn, this provides the desired result. \qed
\end{proof}

As it happens, it is possible to demonstrate for each of the eight strategies in Penney's game that the corresponding optimal strategy \emph{is} optimal using a Markov Chain argument. However, in order to demonstrate a broader range of mathematics, we do not take that approach here.

\begin{figure}
	\centering
	\begin{tikzpicture}
	    \node (HHT) at (-1, 0) {$\set{HHT}$};
	    \node (HTT) at ( 1, 1) {$\set{HTT}$};
	    \node (TTH) at ( 3, 0) {$\set{TTH}$};
	    \node (THH) at ( 1,-1) {$\set{THH}$};
	    \node (HHH) at (-1,-1) {$\set{HHH}$};
	    \node (HTH) at (-1, 1) {$\set{HTH}$};
	    \node (TTT) at ( 3, 1) {$\set{TTT}$};
	    \node (THT) at ( 3,-1) {$\set{THT}$};
    
	    \begin{scope}[every path/.style={->}]
	    	\draw (HTT) -- (HHT); 
	       	\draw (HHT) -- (THH);
	        \draw (THH) -- (TTH);
	        \draw (TTH) -- (HTT);
        
    	    \draw (TTT) -- (HTT);
    	    \draw (HTH) -- (HHT);
    	    \draw (HHH) -- (THH);
    	    \draw (THT) -- (TTH);
      	
   		\end{scope}  
	\end{tikzpicture}
    
	\caption{A graphical representation showing the relationship between the first player's chosen strategy and the strategy elected by the second player. An arrow $\left(\rightarrow\right)$ indicates that, given that the first player selected the source sequence, the second player should optimally play for the destination sequence.}
    \label{fig:optimal}
\end{figure}
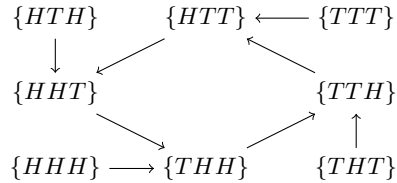

\begin{proof}[Part Two of \Cref{theorem:optimal}]
	We have already demonstrated proofs that \Cref{theorem:optimal} yields an optimal strategy for the cases where \play chooses either of the 3-sequences $\set{HHH}$ or $\set{HHT}$. By \Cref{lemma:markov_penney}, we have immediately that \Cref{theorem:optimal} gives the correct strategy for the cases where \play elects either of the strategies $\set{HTH}$ or $\set{HTT}$. 
    
    The remainder of the proof is not much work. In order to avoid the tedious repetition inherent in enumerating the remaining four cases, one can get away with simply observing that we could have arbitrarily relabeled the sides of our coin so that our ``heads'' side becomes our $T$ and the ``tails'' face represents $H$. Therefore, we may argue the last four cases via symmetry as follows.
    \begin{description}
    	\item[$\set{TTT}$:] By symmetry and the result in \Cref{lemma:hhh_response}, we have that $\set{HTT}$ is optimal for $\set{TTT}$.
        \item[$\set{TTH}$:] By the first part of the proof for \Cref{theorem:optimal}, $\set{HTT}$ is an optimal strategy.
        \item[$\set{THT}$:] By \Cref{lemma:markov_penney} we have by symmetry that $\set{TTH}$ is optimal for $\set{THT}$.
        \item[$\set{THH}$:] By the same logic as in the prior point, $\set{TTH}$ is also optimal for the strategy $\set{THH}$.
    \end{description}
    This completes the proof that \Cref{theorem:optimal} does indeed produce an optimal strategy for \opp given any of the eight strategies that \play can elect. \Cref{fig:optimal} provides a visual reference for understanding the relationships between the strategies in the the Penney ante game. \qed
\end{proof}

\section{A Markov Chain Analysis of the Penney Ante Game}

\begin{definition}[Absorbing Chains]
	Let \P~be the transition matrix of a Markov Chain. We say that a state $s_i$ is absorbing if $\P_{i,i} = 1$ such that the state may never be left. States that are not absorbing are called transient. A Markov Chain is itself absorbing if it contains at least one such state and if it is possible to transition to one of those states from any of the other states. 
\end{definition}

\begin{definition}[Markov Canonical Form]
	Let \P~be the transition matrix for an absorbing Markov Chain. Then it is possible to ``shuffle around'' the states of \P~such that \P~has the form,
    \begin{equation}
    	\P = \left(\begin{array}{ccc}
        		& \text{Transitive} & \text{Absorbing} \\
                  \text{Transitive} & \Q & \R \\
                  \text{Absorbing}  & \mbox{\boldmath $0$} & \mbox{\boldmath $I$}
        	\end{array}\right)
    \end{equation}
    Let the matrix have $n_{\text{absorbing}}$ absorbing states and $n_{\text{transient}}$ transient states. Then \Q~is a $n_{\text{transient}}\times n_{\text{transient}}$ matrix, and \R~is a $n_{\text{transient}}\times n_{\text{absorbing}}$ matrix. As usual, $\mbox{\boldmath $I$}$ indicates a $n_{\text{absorbing}}\times n_{\text{absorbing}}$ identity matrix and $\mbox{\boldmath $0$}$ is a $n_{\text{absorbing}}\times n_{\text{transient}}$ matrix of zeros. Our analysis relies primarily on the matrix \Q.
\end{definition}

\begin{theorem}[The Fundamental Matrix]
	\label{theorem:fundamental}
	Let \P~be the transition matrix for a Markov Chain in canonical form. Then there exists a matrix \N called the fundamental matrix which is,
    \begin{equation}
    	\N = \arg{\mbox{\boldmath $I$} - \Q}^{-1} = \mbox{\boldmath $I$} + \sum_{i=1}^\infty \Q^i
    \end{equation}
    Furthermore, the entry $\N_{i,j}$ is the expected number of times that a Markov process with transition matrix \P~visits the transient state $s_j$ conditioned on the fact that it began in transient state $s_i$.
    
    For a proof of this theorem, refer to \cite{probability}.
\end{theorem}

\subsection{How Many Flips on Average?}

\begin{theorem}[Expected Time to Victory]
	\label{theorem:expectation}
	Leveraging the idea that Penney's game may be viewed in the light of Markov Chains, we are to compute the expected number of rounds that will be played before a winner is determined in the case of a single trick. Let $\mbox{\boldmath $e$}$ be a vector whose $i^{\text{th}}$ entry is the expected number of iterations of the Markov process beginning in state $s_i$ undergoes before it enters an absorbing state. Then we may write that,
    \begin{equation}
    	\mbox{\boldmath $e$} = \N \left\{1,1,\ldots, 1,1\right\}^T
    \end{equation}
\end{theorem}

\begin{proof}
	By \Cref{theorem:fundamental} we have that \N~represents the expected number of times that a process beginning in any of the transient states visits each of the other transient states. Thus, taking the summation of the rows of \N~yields the expected number of times that the process beginning in state $s_i$ is in any of the other transient states. This is the expected number of iterations before that process is absorbed. Notice that $\N \left\{1,1,\ldots, 1,1\right\}^T$ is precisely the summation of the rows expressed as a matrix multiplication. \qed
\end{proof}

\vspace{-5mm}

\begin{table}
\begin{center}
	\begin{tabular}{c|cccccccc}
		\hline
		\multicolumn{9}{c}{~~~~~~~~~~~~~~~~~~~~~~~~Initial 3-Sequences} \\
		\hline 
        & $\set{\scriptstyle{HHH}}$ & $\set{\scriptstyle{HHT}}$ & $\set{\scriptstyle{HTH}}$ & $\set{\scriptstyle{HTT}}$ & $\set{\scriptstyle{THH}}$ & $\set{\scriptstyle{THT}}$ & $\set{\scriptstyle{TTH}}$ & $\set{\scriptstyle{TTT}}$\\
        $\arg{\set{\scriptstyle{HHH}}, \set{\scriptstyle{THH}}}$ & 0 & 6 & 4 & 6 & 0 & 6 & 4 & 6 \\
        $\arg{\set{\scriptstyle{HHT}}, \set{\scriptstyle{THH}}}$ & 2 & 0 & 4 & 6 & 0 & 6 & 4 & 6 \\
        $\arg{\set{\scriptstyle{HTH}}, \set{\scriptstyle{HHT}}}$ & 2 & 0 & 0 & 6 & 2 & 4 & 4 & 6 \\
        $\arg{\set{\scriptstyle{HTT}}, \set{\scriptstyle{HHT}}}$ & 2 & 0 & $3\frac{1}{3}$ & 0 & 2 & $2\frac{2}{3}$ & $3\frac{1}{3}$ & $5\frac{1}{3}$ \\
        $\arg{\set{\scriptstyle{THH}}, \set{\scriptstyle{TTH}}}$ & $5\frac{1}{3}$ & $3\frac{1}{3}$ & $2\frac{2}{3}$ & 2 & 0 & $3\frac{1}{3}$ & 0 & 2 \\
        $\arg{\set{\scriptstyle{THT}}, \set{\scriptstyle{TTH}}}$ & 6 & 4 & 4 & 2 & 6 & 0 & 0 & 2 \\
        $\arg{\set{\scriptstyle{TTH}}, \set{\scriptstyle{HTT}}}$ & 6 & 4 & 6 & 0 & 6 & 4 & 0 & 2 \\
        $\arg{\set{\scriptstyle{TTT}}, \set{\scriptstyle{HTT}}}$ & 6 & 4 & 6 & 0 & 6 & 4 & 6 & 0 \\
		\hline
	\end{tabular}
    \vspace{2mm}
    \caption{We present here the expected time until absorption for each of the eight possible 3-sequence strategies that \play can elect and assuming that \opp will play optimally according to \Cref{theorem:optimal}. We denote by the notation $\arg{s_1,s_2}$ for strategies $s_1$ and $s_2$ that \play elected $s_1$ and that the optimal response is $s_2$. Notice that it it clear that if a process begins in either of $s_1$ or $s_2$ then the game is ended immediately and the expected time until absorption is zero.} 
    
    \label{tab:expected_absorption} 
\end{center}
\end{table}

\vspace{-1cm}

We have now established the mathematical foundation that will allow us to present our main result. We calculate the expected number of flips of a coin that will occur before a winner is declared in Penney's game. We assume that given the first player's sequence the second player will choose their strategy according to \Cref{theorem:optimal}. This is, of course, dependent on knowing which of the eight possible initial circumstances ended up occurring. We exhibit these results in \Cref{tab:expected_absorption}. We wish to draw attention to interesting symmetry that exists in the tabulation of the results that results from the symmetric nature of the game.

\begin{theorem}[Expected Length of Penney's Game]
	\label{theorem:expected_length}
	For an instance of Penney's game in which \play assumes the strategy $s_1$ and \opp takes the optimal strategy as in \Cref{theorem:optimal}, then the expected number of coin flips before a winner is declared is,
    \begin{equation}
    	\frac{\arg{0 + 3}}{8} + \frac{\arg{0 + 3}}{8} + \frac{1}{8}\sum_{i=1}^6 \arg{\mbox{\boldmath $e$}^{\arg{s_1,s_2}}_i + 3}
    \end{equation}
\end{theorem}
\begin{proof}
	Consider that each of the initial states are equally probable, each occurring with frequency equal to $\frac{1}{8}$. Then the average time to absorption for a process beginning in either of $s_1$ or $s_2$ is clearly zero. Additionally, the vector \mbox{\boldmath $e$} presents the average time to absorption for each of the six ``non-immediate-victory'' states. Therefore, if we do not condition on any particular initial sequence at the beginning of the game, we may simply average together all of the expected times to absorption.
    
    However, it is important to note that in this case averaging these expected times to absorption will not yield the expected number of coins flipped before a winner is declared. This is because we are assumed to start with a given 3-sequence in each of these cases, yet it takes three coin flips to arrive there in the first place! Hence, we introduce a supplement of three to each of our expected times to absorption to obtain the expected number of coins we must flip. \qed

\end{proof}

\vspace{-0.6cm}

\begin{table}
\begin{center}
	\begin{tabular}{|c|c|c|c|}
		\hline
		\multicolumn{4}{c}{Possible Penney Ante Games} \\
		\hline 
        $\arg{\set{\scriptstyle{HHH}}, \set{\scriptstyle{THH}}}$ &
        $\arg{\set{\scriptstyle{HHT}}, \set{\scriptstyle{THH}}}$ &
        $\arg{\set{\scriptstyle{HTH}}, \set{\scriptstyle{HHT}}}$ &
        $\arg{\set{\scriptstyle{HTT}}, \set{\scriptstyle{HHT}}}$ \\\hline
        7 & $6\frac{1}{2}$ & 6 & $5\frac{1}{3}$ \\
        \hline\hline
        $\arg{\set{\scriptstyle{THH}}, \set{\scriptstyle{TTH}}}$ &
        $\arg{\set{\scriptstyle{THT}}, \set{\scriptstyle{TTH}}}$ &
        $\arg{\set{\scriptstyle{TTH}}, \set{\scriptstyle{HTT}}}$ & 
        $\arg{\set{\scriptstyle{TTT}}, \set{\scriptstyle{HTT}}}$ \\\hline
        $5\frac{1}{3}$ & 6 & $6\frac{1}{2}$ & 7 \\
		\hline
	\end{tabular}
    \vspace{2mm}
    \caption{We present here the expected number of coin tosses that are required for each optimal variation of the Penney coin tossing game according to \Cref{theorem:expected_length}. Notice again the symmetry that results.} 
    
    \label{tab:expected_coins} 
\end{center}
\end{table}

\vspace{-1.1cm}

This represents an approach that can serve as an alternative to the method using Conway numbers proposed by Nishiyama in \cite{nishiyama}. We show the expected number of coin tosses for each of the eight possible games in \Cref{tab:expected_coins}. If we assume that each of the strategies that \play can assume are equally likely, then we have that the expected number of coins flipped in Penney's ante game is given by,
\begin{equation}
	\frac{2}{8} \arg{7 + \frac{13}{2} + 6 + \frac{16}{3}} = \frac{149}{24} \approx 6.2083
\end{equation}

\subsection{An Alternative Derivation of Optimality}

It is possible to derive the probability $\Pr{s_1 \left|\right. s_2}$ for any pair of strategies arising from Penney's game. Here, we present a method that uses the fundamental matrix and another theorem from Markov Chain theory to derive these same probability values. But first, we present a critical result.

\begin{theorem}
	Let \B~be a $n_\text{transient}\times n_{\text{absorbing}}$ matrix such that the value of $\B_{i,j}$ is the probability that a Markov process will be absorbed into the absorbing state $s_j$ given that it started in the transient state $s_i$. Then we may calculate \B~as,
    \begin{equation}
    	\B = \N\R
    \end{equation}
\end{theorem}

\begin{proof}
	This proof is straightforward and requires only some simple manipulations, particularly switching the order of the summations.
	\begin{eqnarray}
		\B_{i,j} & = & \sum_{n=0}^\infty \sum_{k=0}^{n_{\text{transient}}} \Q^n_{i,k} \R_{k,j} = \sum_{k=0}^{n_{\text{transient}}} \sum_{n=0}^\infty \Q^n_{i,k} \R_{k,j} \\
        & = & \sum_{k=0}^{n_{\text{transient}}} \N_{i,k}\R_{k,j} = \arg{\N\R}_{i,j}
    \end{eqnarray}
\end{proof}

\begin{proposition}
	\label{prop:alternative}
	The values $\Pr{s_1 \left|\right. s_2}$ and $\Pr{s_2 \left|\right. s_1}$ may be written as follows,
    \begin{eqnarray}
    	\Pr{s_1 \left|\right. s_2} & = & \frac{1}{8} + \frac{0}{8} + \frac{1}{8}\sum_{i=1}^6 \B_{i, 1} \\
        \Pr{s_2 \left|\right. s_1} & = & \frac{0}{8} + \frac{1}{8} + \frac{1}{8}\sum_{i=1}^6 \B_{i, 2} = 1 - \Pr{s_1 \left|\right. s_2}
    \end{eqnarray}
\end{proposition}

\begin{proof}
	Consider that if the game begins in the 3-sequence $s_1$, then the probability of being absorbed into that state is one, yielding a probability of zero for the event that the process is absorbed into $s_2$. The matrix $\B_{i, 1}$ gives the probability that the process is absorbed into $s_1$ given that it starts in any of the transient states. 
    
    Therefore, given that all eight initial states are equally likely, we have that the probability of the process terminating in $s_1$ is the average of the probabilities that it terminates in $s_1$ given an initial state. This yields the equation,
    \begin{equation}
    	\Pr{s_1 \left|\right. s_2} = \frac{1}{8} + \frac{0}{8} + \frac{1}{8}\sum_{i=1}^6 \B_{i, 1}
    \end{equation}
    To see that the equation for $\Pr{s_2 \left|\right. s_1}$ must hold in \Cref{prop:alternative} we need only confirm that $s_2 \vert s_1$ is a complementary event to $s_1\vert s_2$. This is easily seen to be true and so we obtain,
    \begin{eqnarray}
    	\Pr{s_2 \left|\right. s_1} & = & 1 - \Pr{s_1 \left|\right. s_2} = 1 - \arg{\frac{1}{8} + \frac{0}{8} + \frac{1}{8}\sum_{i=1}^6 \B_{i, 1}}\\
        & = & \frac{7}{8} - \frac{1}{8}\sum_{i=1}^6 \B_{i,1} = \frac{7}{8} - \frac{1}{8} \sum_{i=1}^6 \arg{1 - \B_{i,2}}\\
        & = & \frac{7}{8} - \frac{1}{8}\sum_{i=1}^6 1 + \sum_{i=1}^6 \B_{i,2}=\frac{1}{8} + \sum_{i=1}^6 \B_{i,2}
    \end{eqnarray}
    This completes the proof. \qed
\end{proof}

\end{document}